\documentclass[reqno,12pt]{amsart}
\usepackage{amsmath,amsthm,amsfonts,amssymb}
\usepackage{euscript}
\usepackage[utf8]{inputenc}
\usepackage[english]{babel}
\usepackage{graphicx}
\usepackage{pgf,tikz}
\usepackage[colorlinks, citecolor=blue]{hyperref}
\usetikzlibrary{arrows}
\date{}
\usepackage[letterpaper,margin=0.9in,top=1.25in,bottom=1.15in]{geometry}
\usepackage{ulem}
\usepackage{placeins}
\usepackage{mathrsfs}
\usepackage[]{todonotes}

\def\Ric{{\mathrm {Ric}}}

\newtheorem{theorem}{Theorem}[section]
\newtheorem{lemma}[theorem]{Lemma}
\newtheorem{corollary}[theorem]{Corollary}
\newtheorem{proposition}[theorem]{Proposition}

\theoremstyle{definition}
\newtheorem{definition}[theorem]{Definition}

\theoremstyle{remark}
\newtheorem{remark}[theorem]{Remark}
\numberwithin{equation}{section}

\begin{document}
\title[Counting ends on shrinkers]
{Counting ends on shrinkers}

\author{Jia-Yong Wu}
\address{Department of Mathematics, Shanghai University, Shanghai 200444, China}
\email{wujiayong@shu.edu.cn}

\thanks{}


\dedicatory{}

\date{\today}

\begin{abstract}
In this paper we apply a geometric covering method to study the number
of ends on shrinkers. On one hand, we prove that the number of ends on
any complete non-compact shrinker is at most polynomial growth with
fixed degree. On the other hand, we prove that any complete non-compact
shrinker with certain volume comparison condition has finitely many ends.
Some special cases of shrinkers are also discussed.
\end{abstract}
\maketitle


\section{Introduction and main results}

An $n$-dimensional Riemannian manifold $(M,g)$ is called a \textit{gradient shrinking Ricci
soliton or shrinker} (see \cite{[Ham]}) if there exists a smooth function $f$ on $(M,g)$
such that the Ricci curvature $\text{Ric}$ and the Hessian of $f$ satisfy
\[
\Ric+\mathrm{Hess}\,f=\lambda g
\]
for some constant $\lambda>0$. Function $f$ is often called a \textit{potential}
of the shrinker.  Upon scaling the metric $g$ by a constant, we may assume
$\lambda=1/2$ so that
\begin{align}\label{Eq1}
\Ric +\mathrm{Hess}\,f=\frac 12g.
\end{align}
Furthermore, we can normalize $f$ such that
\eqref{Eq1} simultaneously satisfies
\begin{equation}\label{condition}
\mathrm{S}+|\nabla f|^2-f=0,
\end{equation}
where $\mathrm{S}$ is the scalar curvature of $(M,g)$, and
\begin{equation}\label{condmu}
\int_M (4\pi)^{-\frac n2}e^{-f} dv=e^{\mu},
\end{equation}
where $dv$ is the volume element with respect to metric $g$, and
$\mu=\mu(g,1)$ is the entropy functional of Perelman \cite{[Pe]}. By
Lemma 2.5 in \cite{[LLW]}, we see that the term $e^{\mu}$ is almost equivalent
to the volume of geodesic ball $B(p,1)$ with radius $1$ and center $p$. Here
$p\in M$ is a infimum point of $f$, which can be always achieved for any
complete shrinker; see \cite{[HaMu]}.

Shrinkers play an important role in the Ricci flow as they correspond to some
self-similar solutions and usually arise as the limit solutions of type I
singularity models of the Ricci flow \cite{[EMT]}. They are regarded as a
natural extension of Einstein manifolds with positive scalar curvature,
and are related to the Bakry-\'Emery Ricci tensor \cite{[BE]}. Nowadays,
the understanding of geometry and topology for shrinkers is an important
subject in the Ricci flow \cite{[Ham]}. For dimensions 2 and 3, the
classification of shrinkers is complete. However dimensions equal to or
greater than 4, the complete classification remains open; see
\cite{[Cao1],[Cao2]} and references therein for nice surveys.

It is an interesting phenomenon that many geometric and analytic properties of
shrinkers are similar to manifolds with nonnegative Ricci curvature or Einstein
manifolds with positive scalar curvature. Some interesting results are exhibited
as follows. Wylie \cite{[Wy]} proved that any complete shrinker has finite
fundamental group (the compact case due to Derdzi\'nski \cite{[De]}). Fang, Man
and Zhang \cite{[FMZ]} showed that any non-compact shrinker with bounded scalar
curvature has finite topological type. Chen and Zhou \cite{[CaZh]} confirmed
that any non-compact shrinker has at most Euclidean volume growth. Munteanu and
Wang \cite{[MuWa12]} proved that any non-compact shrinker has at least linear
volume growth.

Haslhofer and M\"uller \cite{[HaMu],[HaMu2]} proved a Cheeger-Gromov compactness
theorem of shrinkers with a lower bound on their entropy and a local integral
Riemann bound. Li, Li and Wang \cite{[LLW]} gave a structure theory for
non-collapsed shrinkers, which was further developed by Huang, Li and Wang
\cite{[HLW]}. For the $4$-dimensional case, Li and Wang \cite{[LiWa19]} proved
that any nontrivial flat cone cannot be approximated by smooth shrinkers with
bounded scalar curvature and Harnack inequality under the pointed-Gromov-Hausdorff
topology. Huang \cite{[Hua]} applied the strategy of Cheeger-Tian \cite{[CT]}
in Einstein manifolds and proved an $\epsilon$-regularity theorem for
$4$-dimensional shrinkers, confirming a conjecture of Cheeger-Tian \cite{[CT]}.

Recently, Li and Wang \cite{[LiWa]} obtained a sharp logarithmic Sobolev inequality,
the Sobolev inequality, heat kernel estimates, the no-local-collapsing theorem,
the pseudo-locality theorem, etc. on complete shrinkers, which can be further
extended to the other geometric inequalities, such as Nash inequalities, Faber-Krahn
inequalities and Rozenblum-Cwikel-Lieb inequalities in \cite{[Wu]}. For more function
theory on shrinkers, the interested readers are referred to
\cite{[GZ],[MSW],[MuW14],[MuWa14e], [Wu15], [WW15],[WW16]} and references therein.

On a manifold $M$, a set $E$ is called an \textit{end} with respect to a compact set
$\Omega\subset M$, if it is an unbounded connected component of $M\backslash\Omega$.
The number of ends with respect to $\Omega$, denoted by $N_\Omega(M)$, is the number
of unbounded connected components of $M\backslash\Omega$. If $\Omega_1\subset\Omega_2$,
then $N_{\Omega_1}(M)\le N_{\Omega_2}(M)$. Hence if $\Omega_i$ is a compact exhaustion
of $M$, then $N_{\Omega_i}(M)$ is a nondecreasing sequence. If this sequence is bounded,
then we say that $M$ has finitely many ends. In this case, the number of ends of $M$ is
defined by
\[
N(M)=\lim_{i\to\infty}N_{\Omega_i}(M).
\]
Obviously, the number of ends is independent of the compact exhaustion $\{\Omega_i\}$.
Ends of manifolds are related to the geometry and topology of manifolds;
the interested reader may refer to the book \cite{[PL]}.

The Cheeger-Gromoll's splitting theorem \cite{[CG]} indicates that any
complete non-compact manifold with nonnegative Ricci curvature has at most two ends.
Later, Cai \cite{[Cai]} and Li-Tam \cite{[LT]} independently proved that any manifold
with nonnegative Ricci curvature outside a compact set has at most finitely many ends
(see also Liu \cite{[Liu2]}); see \cite{[Wu16]} for an extension to smooth metric
measure spaces. Cai's approach is pure geometrical, strongly depending on a local
version of Cheeger-Gromoll's splitting theorem, while Li-Tam's proof is analytic in
nature by taking full advantage of the harmonic function theory. Liu's proof is also
geometrical, not adapting the local splitting theorem but using various volume
comparisons. At present, an interesting question of whether the Cheeger-Gromoll
splitting theorem holds on any complete non-compact shrinker still remains unresolved.
In the next attempt to consider the number of ends, it is natural to ask

\vspace{.1in}

\noindent \textbf{Question}. \textit{Does any a complete non-compact shrinker
have finitely many ends?}

\vspace{.1in}

For the K\"ahler case, Munteanu and Wang \cite{[MuWa14e]} proved that any
K\"ahler shrinker has only one end. For the Riemannian case, Munteanu, Schulze
and Wang \cite{[MSW]} showed that the number of ends is finite when the
scalar curvature satisfies certain scalar curvature integral at infinity.
Their proof depends on the Li-Tam's analytic theory \cite{[LT]}.
In this paper, we use a geometric covering argument and prove that
\begin{theorem}\label{endest}
The number of ends on $n$-dimensional complete non-compact shrinker with
the scalar curvature
\[
\mathrm{S}\ge \delta
\]
for some constant $\delta\ge 0$ is at most polynomial growth with degree $2(n-\delta)$.
\end{theorem}

\begin{remark}
From \eqref{en1} in Section \ref{volcom}, we will see that $\mathrm{S}\ge\delta$
implies $\delta\le n/2$ on shrinkers. From Remark \ref{levcon}, we have that
the point-wise assumption $\mathrm{S}\ge \delta$ can be replaced by a lower 
of the average scalar curvature over the level set 
$\{f<r\}:=\left\{x\in M|f(x)<r\right\}$ for any $r>0$, that is,
\[
\frac{1}{\int_{\{f<r\}}dv}\int_{\{f<r\}}\mathrm{S}\,dv\ge\delta
\]
for any $r>0$. If the scalar curvature also has a uniformly upper
bound, then the degree $2(n-\delta)$ in theorem can be reduced to $n-2\delta$;
see Remark \ref{reN2}.
\end{remark}

The following condition introduced in \cite{[LT2]} will play an important role
in this paper.
\begin{definition}
A Riemannian manifold $(M,g)$ has \textit{volume comparison condition} if there exists
a constant $\eta>0$ such that for all $r\ge r_0$ for some $r_0>0$,
and all $x\in\partial B(q,r)$,
\[
\mathrm{Vol}(B(q,r))\le\eta\,\mathrm{Vol}\left(B(x,\frac{r}{16})\right),
\]
where $\mathrm{Vol}(B(q,r))$ is the volume of geodesic ball $B(q,r)$ of radius $r$
with center at a fixed point $q\in M$.
\end{definition}

If the shrinker satisfies volume comparison condition, we prove that
\begin{theorem}\label{main1}
Any complete non-compact shrinker with volume comparison condition must have finitely many ends.
\end{theorem}

Many special cases of shrinkers satisfy volume comparison condition. The detailed
discussion can be referred to Section \ref{sec4}. Here we summarize
some results as follows:

(I) If a manifold satisfies volume doubling property, then it admits volume
comparison condition; see Proposition \ref{voldoub}. Recall that $(M,g)$
is said to be \textit{volume doubling property} if
\[
\mathrm{Vol}(B(x,2r))\le D\,\mathrm{Vol}(B(x,r))
\]
for any $x\in M$ and $r>0$, where $D$ is a fixed constant. Clearly, any
manifold with nonnegative Ricci curvature satisfies volume doubling property.

(II) If the asymptotic scalar curvature ratio of shrinker is finite, then such
shrinker has volume comparison condition; see Proposition \ref{decc}. Given a
point $q\in (M,g)$, the \textit{asymptotic scalar curvature ratio}
($\operatorname{ASCR}$) is defined by
\[
\operatorname{ASCR}(g):=\underset{r(q,x)\to\infty}{\lim\sup}\,\mathrm{S}(x)\cdot r(q,x)^2,
\]
$r(q,x)$ is the distance function from $q$ to $x$. It is easy to see that
$\operatorname{ASCR}(g)$ is independent of the base point $q$. Chow,
Lu and Yang \cite{[ChLY]} proved that a non-compact non-flat shrinker has
at most quadratic scalar curvature decay. Therefore, except the flat
shrinker, our assumption is in fact equivalent to $\operatorname{ASCR}(g)=c_0$
for some constant $c_0>0$, which takes place at least for the asymptotically
conical shrinker \cite{[KW]}.

(III) If a family of average of scalar curvature integral has at least
quadratic decay of radius, precisely, for a infimum point $p\in M$ of $f$,
there exists a constant $c_1>0$ such that
\[
\frac{r^2}{\mathrm{Vol}\left(B(x,r)\right)}\int_{B(x,r)}\mathrm{S}\,dv\le c_1
\]
for all $r>0$ and all $x\in\partial B(p,r)$, then such shrinker has volume
comparison condition; see Proposition \ref{intevc}. The class of average scalar
curvature integral can be regarded as some energy functions of scalar curvature,
which is derived from Li-Wang (logarithmic) Sobolev inequalities; see Lemma
\ref{logeq2} or Lemma \ref{slogeq}.

(IV) If a complete non-compact shrinker $(M,g,f)$ with a infimum point $p\in M$ of $f$
satisfies
\[
\mathrm{Vol}\left(B(x,\frac{r}{16})\right)\ge c_2\,r^n
\]
for all $r>0$ and all $x\in\partial B(p,r)$, where $c_2$ is a positive
constant, then such shrinker satisfies volume comparison condition; see
Corollary \ref{AVRc}. This condition can be regarded as a family of
Euclidean volume growth, which seems to be stronger than the positive
asymptotic volume ratio; see the end of Section \ref{sec4} for the
detailed discussion.

Besides, Li and Tam \cite{[LT2]} proved that if a Riemannian manifold with
each end has asymptotically non-negative sectional curvature, then it satisfies
the volume comparison condition. Recall that $(M,g)$ has \textit{asymptotically
non-negative sectional curvature} if there exists a point $q\in M$ and a
continuous decreasing function $\tau:\mathbb{R}^{+}\to\mathbb{R}^{+}$ such that
$\int^{+\infty}_0t\tau(t)\,dt<\infty$
and the sectional curvature $K(x)$ at any point $x\in M$ satisfies
$K(x)\ge-\tau(r(q,x))$,
where $r(q,x)$ is a distance function from $q$ to $x$. Li and Tam \cite{[LT2]}
also proved that if a Riemannian manifold with finite first Betti number has
nonnegative Ricci curvature outside a compact set, then it satisfies volume
comparison condition. We refer the readers to \cite{[LT2]} for further related
discussions.

Different from Munteanu-Schulze-Wang's analytic argument, our proof of Theorem
\ref{endest} is geometrical, which stems from Liu's approach \cite{[Liu2]},
but we have a major obstacle due to the lack of volume comparison at different
points and radii. For manifolds with nonnegative Ricci curvature (outside a
compact set), such properties come from classical relative volume comparisons.
With these comparisons, Liu was able to get a ball covering property of
manifolds with nonnegative Ricci curvature (outside a compact set) and hence
proved finitely many ends. But for shrinkers, we only prove relative
volume comparisons about geodesic balls with center at a base point; see
Theorem \ref{relcompar} in Section \ref{volcom}. We do not know if they
could hold for geodesic balls centered at different points. To overcome
this difficulty, we extend Cao-Zhou upper volume bound \cite{[CaZh]}
(further development by Munteanu-Wang \cite{[MuWa12]}, Zhang \cite{[Zh]})
to a more precise statement; see Lemma \ref{logeq1}; while we generalize
the Li-Wang lower volume bound \cite{[LiWa]}; see Lemmas \ref{logeq2}
and \ref{slogeq}. Applying these upper and lower volume estimates, we
could get a weak volume comparison condition; see Proposition \ref{vd}
in Section \ref{sec3}. This proposition is enough to produce a weak ball
covering property (see Theorem \ref{sdest} in Section \ref{sec3}) and
finally leads to Theorem \ref{endest}. In particular, when the shrinker
satisfies volume comparison condition, we can prove Theorem \ref{main1}
in a similar spirit.

The rest of paper is organized as follows. In Section \ref{volcom}, we
will prove upper and lower relative volume comparisons of the shrinker
in geodesic balls with center at a base point. We also give some upper
and lower volume estimates. In Section \ref{sec3}, we will use volume
comparisons of Section \ref{volcom} to prove a weak ball covering
property. Then we apply the weak ball covering property to prove Theorem
\ref{endest}. In Section \ref{sec4}, when the shrinker satisfies volume
comparison condition, we will prove Theorem \ref{main1} by adapting
the argument of Theorem \ref{endest}. Meanwhile, we will provide various
sufficient condition to ensure volume comparison condition. In Section
\ref{sec5}, we will apply the ball covering property of shrinkers to
study the diameter growth of ends.

In the whole of this paper, we let $c(n)$ denote a constant depending only on dimension
$n$ of shrinker $(M,g,f)$ whose value may change from line to line.

\vspace{.1in}

\textbf{Acknowledgements}.
The author thanks Yu Li for his valuable suggestions and stimulating
discussions, which improves some results in this paper. The author
also thanks Guoqiang Wu for his helpful comments on an earlier version
of this paper. Finally the author sincerely thanks Professor Ovidiu
Munteanu for valuable comments and pointing out a mistake of an earlier
version of the paper.

%
\section{Volume comparison}\label{volcom}
In this section, we will discuss upper and lower relative volume comparisons
of shrinker about geodesic balls with center at a base point. We will
also discuss upper and lower volume estimates of shrinkers.

Recall that the potential $f$ of shrinker is uniformly equivalent to the distance function
squared. Precisely, the following sharp estimate was established  originally
due to Cao-Zhou \cite{[CaZh]} and later improved by Haslhofer-M\"uller \cite{[HaMu]};
see also Chow et al. \cite{[Chowetc]}.
\begin{lemma}\label{potenesti}
Let $(M,g, f)$ be an $n$-dimensional complete non-compact shrinker satisfying
\eqref{Eq1} and \eqref{condition}. For any point $q\in M$, $f$ satisfies
\[
\frac 14\left[\left(r(q,x)-2\sqrt{f(q)}-4n+\frac 43\right)_{+}\right]^2\le f(x)\le\frac 14\left(r(q,x)+2\sqrt{f(q)}\right)^2
\]
for all $x\in M$, where $r(q,x)$ denotes a distance function from $q$
to $x$.

Moreover, there exists a point $p\in M $ where $f$ attains its infimum in $M$
such that $f(p)\le n/2$; meanwhile $f$ has a simple estimate
\[
\frac 14\left[\big(r(p,x)-5n\big)_{+}\right]^2\le f(x)\le\frac 14\left(r(p,x)+\sqrt{2n}\right)^2
\]
for all $x\in M$. Here $a_+=\max\{a,0\}$ for $a\in \mathbb{R}$.
\end{lemma}

Chen \cite{[Chen]} proved that the scalar curvature of shrinkers has a lower
bound
\[
\mathrm{S}\ge 0.
\]
Pigola, Rimoldi and Setti \cite{[PiRS]} showed that the scalar curvature
$\mathrm{S}$ is strictly positive, unless $(M,g,f)$ is the Gaussian
shrinking Ricci soliton. By Lemma \ref{potenesti} and \eqref{condition},
the scalar curvature naturally has an upper bound
\begin{equation}\label{scaup}
\mathrm{S}(x)\le\frac 14\left(r(p,x)+\sqrt{2n}\right)^2
\end{equation}
for all $x\in M$. This upper bound will be used in this paper.

Recently, Li and Wang \cite{[LiWa]} applied the monotonicity of Perelman's
functional along Ricci flow and the invariance of Perelman's functional under
diffeomorphism actions to obtain (logarithmic) Sobolev inequalities on
complete shrinkers.
\begin{lemma}\label{sobineq}
Let $(M,g, f)$ be an $n$-dimensional shrinker satisfying \eqref{Eq1},
\eqref{condition} and \eqref{condmu}. Then for any $\varphi\in C^{\infty}_0(M)$
with $\int_M\varphi^2dv=1$ and any $\tau>0$,
\begin{equation}\label{LSI}
\mu+n+\frac n2\ln(4\pi)\le\tau\int_M\left(4|\nabla\varphi|^2+\mathrm{S}\varphi^2\right)dv-\int_M\varphi^2\ln \varphi^2dv-\frac n2\ln \tau.
\end{equation}
Moreover, for any $u\in C^{\infty}_0(M)$,
\begin{equation}\label{sobo}
\left(\int_Mu^{\frac{2n}{n-2}}dv\right)^{\frac{n-2}{n}}\le c(n)e^{-\frac{2\mu}{n}}\int_M\left(4|\nabla u|^2+\mathrm{S}u^2\right) dv.
\end{equation}
\end{lemma}
The above inequalities are useful for understanding the geometry and topology
for shrinkers; see some recent works \cite{[LiWa]}, \cite{[MSW]}, \cite{[Wu21]}
and \cite{[Wu]}. In the following sections, we will apply them to study the
volume growth of shrinkers.

We start to discuss some applications of the above lemmas. First, applying Lemma
\ref{potenesti}, we can provide a relative volume comparison with center at any
a base point for large geodesic balls. Similar volume comparison was ever
considered by Carrillo and Ni \cite{[CaNi]} under some extra assumption.
\begin{theorem}\label{relcompar}
Let $(M,g,f)$ be a shrinker satisfying \eqref{Eq1}. For any point
$q\in M$,
\[
\frac{\mathrm{Vol}(B(q,R))}{\mathrm{Vol}(B(q,r))}\le 2\left(\frac{R+c}{r-c}\right)^n
\]
for all $R\ge r\ge 2\sqrt{n}+c$. In particular, for any $0<\alpha<1$,
\[
\frac{\mathrm{Vol}(B(q,R))}{\mathrm{Vol}(B(q,\alpha R))}\le 2\left(1+\frac{2}{\alpha}\right)^n
\]
for all $R\ge2\alpha^{-1}(\sqrt{n}+c)$. Here $c:=2\sqrt{f(q)}+4n-4/3$.
\end{theorem}
\begin{proof}[Proof of Theorem \ref{relcompar}]
The proof is essentially contained in the argument of Cao and Zhou
\cite{[CaZh]}, and we include it for the completeness. Define
\[
\rho(x):=2\sqrt{f(x)}.
\]
By Lemma \ref{potenesti},
\[
r(q,x)-c\le\rho(x)\le r(q,x)+c,
\]
where $c=2\sqrt{f(q)}+4n-4/3$. Denote by
\[
D(r):=\{x\in M|\rho(x)<r\}\quad  \mathrm{and}\quad V(r):=\int_{D(r)}dv.
\]
We trace \eqref{Eq1} and get
\[
\mathrm{S}+\Delta f=\frac n2.
\]
Integrating this equality and using some properties on shrinkers,
Cao and Zhou \cite{[CaZh]} established the following interesting equality:
\begin{equation}\label{VRrel}
n V(r)-rV'(r)=2\int_{D(r)}\mathrm{S}\,dv-2\int_{\partial D(r)}\frac{\mathrm{S}}{|\nabla f|}dv.
\end{equation}
Letting
\[
\chi(r):=\int_{D(r)}\mathrm{S}\,dv,
\]
then by the co-area formula, \eqref{VRrel} can be rewritten as
\[
n V(r)-rV'(r)=2\chi(r)-\frac{4}{r}\chi'(r),
\]
that is,
\[
(r^{-n}V(r))'=4r^{-n-2}\chi'(r)-2r^{-n-1}\chi(r).
\]
Integrating this from  $r$ to $R$ yields
\begin{equation*}
\begin{aligned}
R^{-n}V(R)-r^{-n}V(r)&=4R^{-n-2}\chi(R)-4r^{-n-2}\chi(r)\\
&\quad+2\int^R_rt^{-n-3}\chi(t)\left(2(n+2)-t^2\right)dt.
\end{aligned}
\end{equation*}
For the last term of the above equality, since $\chi(t)$ is positive
and increasing in $t$, then for any $R\ge r\ge \sqrt{2(n+2)}$, we have
\begin{equation*}
\begin{aligned}
2\int^R_rt^{-n-3}\chi(t)\left(2(n+2)-t^2\right)dt&\le2\chi(r)\int^R_rt^{-n-3}\left(2(n+2)-t^2\right)dt\\
&=2\chi(r)\left(-2t^{-n-2}+\frac{t^{-n}}{n}\right){\bigg|}^R_r\\
&=-4R^{-n-2}\chi(r)+4r^{-n-2}\chi(r)+\frac2n \chi(r)(R^{-n}-r^{-n}).
\end{aligned}
\end{equation*}
Hence,
\[
R^{-n}V(R)-r^{-n}V(r)\le 4R^{-n-2}\left(\chi(R)-\chi(r)\right)+\frac2n \chi(r)(R^{-n}-r^{-n})
\]
for $R\ge r\ge \sqrt{2(n+2)}$. Therefore,
\begin{equation}\label{ineq1}
V(R)\le (r^{-n}V(r))R^n+4R^{-2}\chi(R)
\end{equation}
for all $R\ge r\ge \sqrt{2(n+2)}$.

On the other hand, for any $R\ge2\sqrt{n}$, we have
\begin{equation}\label{ineq2}
4R^{-2}\chi(R)\le 2nR^{-2}V(R)\le \frac 12 V(R).
\end{equation}
Substituting \eqref{ineq2} into \eqref{ineq1} gives
\[
\frac{V(R)}{V(r)}\le 2\left(\frac Rr\right)^n
\]
for any $R\ge r\ge 2\sqrt{n}(\ge \sqrt{2(n+2)})$. This implies
\[
\frac{V(R+c)}{V(r-c)}\le 2\left(\frac{R+c}{r-c}\right)^n
\]
for $R\ge r\ge 2\sqrt{n}+c$, where $c:=2\sqrt{f(q)}+4n-4/3$. We also notice
\[
\mathrm{Vol}(B(q,R))\le V(R+c) \quad\mathrm{and}\quad
\mathrm{Vol}(B(q,r))\ge V(r-c)
\]
for any $R\ge0$ and $r\ge c$. Therefore,
\[
\frac{\mathrm{Vol}(B(q,R))}{\mathrm{Vol}(B(q,r))}\le 2\left(\frac{R+c}{r-c}\right)^n
\]
for $R\ge r\ge 2\sqrt{n}+c$, which proves the first part of theorem.

In particular, we choose $r=\alpha R$, where $0<\alpha<1$
and the above estimate becomes
\[
\frac{\mathrm{Vol}(B(q,R))}{\mathrm{Vol}(B(q,\alpha R))}\le 2\left(\frac{R+c}{\alpha R-c}\right)^n
\]
for $R\ge\alpha^{-1}(2\sqrt{n}+c)$. Furthermore, we let
$\alpha R-c>\frac{\alpha}{2}R$, that is, $R\ge2\alpha^{-1}c$, then
\[
\frac{\mathrm{Vol}(B(q,R))}{\mathrm{Vol}(B(q,\alpha R))}\le 2\left(1+\frac{2}{\alpha}\right)^n
\]
for $R\ge2\alpha^{-1}(\sqrt{n}+c)$. This finishes the second part of theorem.
 \end{proof}

Second, following the argument of \cite{[CaZh]}, we can apply Lemma \ref{potenesti}
to give a reverse relative volume comparison.
\begin{theorem}\label{relcompar2}
Let $(M,g,f)$ be a shrinker with a base point $q\in M$ satisfying
\eqref{Eq1}. If the scalar curvature $\mathrm{S}\le\sigma$ for some
constant $0<\sigma<n/2$, then
\[
\frac{\mathrm{Vol}(B(q,R))}{\mathrm{Vol}(B(q,r))}\ge\left(\frac{R-c}{r+c}\right)^{n-2\sigma}
\]
for all $R\ge r+2c$ and $r>0$, where $c:=2\sqrt{f(q)}+4n-4/3$.
\end{theorem}
\begin{proof}[Proof of Theorem \ref{relcompar2}]
By \eqref{VRrel}, $\mathrm{S}\ge 0$ and our curvature assumption $\mathrm{S}\le\sigma$, we have
\[
(n-2\sigma) V(t)\le tV'(t)
\]
for any $t\ge 0$. Integrating this inequality from $r$ to $R$, we get
\[
\frac{V(R)}{V(r)}\ge\left(\frac{R}{r}\right)^{n-2\sigma}
\]
for any $R\ge r>0$. We also see that
\[
\mathrm{Vol}(B(q,r))\le V(r+c) \quad\mathrm{and}\quad
\mathrm{Vol}(B(q,R))\ge V(R-c)
\]
for any $r\ge0$ and $R\ge c$. Therefore,
\[
\frac{\mathrm{Vol}(B(q,R))}{\mathrm{Vol}(B(q,r))}\ge\frac{V(R-c)}{V(r+c)}\ge\left(\frac{R-c}{r+c}\right)^{n-2\sigma}
\]
for any $R\ge r+2c$ and $r>0$.
\end{proof}

Next we will discuss some volume estimates of geodesic balls on shrinkers. The
sharp upper volume estimate was first proved by Cao-Zhou (see Theorem 1.2 in
\cite{[CaZh]}), later a explicit coefficient was stated by Munteanu-Wang (see
Theorem 1.4 in \cite{[MuW14]}) by using a delicate generalized Laplace comparison.
Furthermore, Zhang \cite{[Zh]} proved a sharp quantitative upper volume of the
shrinker with scalar curvature bounded below; see also \cite{[Chowetc]}.
In the following we will improve previous upper volume estimates when $r$ is
not large.
\begin{lemma}\label{logeq1}
Let $(M,g, f)$ be an $n$-dimensional complete non-compact shrinker satisfying
\eqref{Eq1}, \eqref{condition} and \eqref{condmu}. For any point $q\in M$
and for all $r\ge 0$,
\[
\mathrm{Vol}(B(q,r))\le c(n)e^{f(q)}r^n.
\]
Moreover, if the scalar curvature $\mathrm{S}\ge\delta$
for some constant $\delta\ge 0$, then
\[
\mathrm{Vol}(B(q,r))\le c(n)e^{f(q)}e^{-\frac{\delta}{r^2}}\,r^{n-2\delta}
\]
for all $r\ge 2\sqrt{n+2}+c$, where $c:=2\sqrt{f(q)}+4n-4/3$; in particular,
if $p\in M$ is a infimum point of $f$, then
\[
\mathrm{Vol}(B(p,r))\le c(n)e^{-\frac{\delta}{r^2}}\,r^{n-2\delta}
\]
for all $r\ge c(n)$.
\end{lemma}

\begin{proof}[Proof of Lemma \ref{logeq1}]
The first estimate is Theorem 1.4 in \cite{[MuW14]}. So we only need to prove the
second and third estimates. We remark that the second estimate with a rough
coefficient has been proved by Zhang \cite{[Zh]} (see also \cite{[Chowetc]}).
Here, we need to figure out the accurate coefficients, which plays a key role
in our application.

For convenience of our computation, we adapt the notations of \cite{[Zh]}
(see also \cite{[Chowetc]}), which are sight different from those in \cite{[CaZh]}.
For any $t\in\mathbb{R}$, let 
\[
\{f<t\}:=\left\{x\in M|f(x)<t\right\}
\]
and define
\[
\mathscr{V}(t):=\int_{\{f<t\}}dv\quad \mathrm{and} \quad
\mathscr{R}(t):=\int_{\{f<t\}}\mathrm{S}dv.
\]
Notice that for any $q\in M$, $f(x)$ satisfies
\[
\frac 14\left[(r(x,q)-c)_{+}\right]^2\le f(x)\le\frac 14\left(r(x,q)+c\right)^2
\]
where $c:=2\sqrt{f(q)}+4n-4/3$. Therefore, if $r\ge c$, then
\[
\left\{f<\frac 14(r-c)^2\right\}\subset B(q,r)\subset\left\{f<\frac 14(r+c)^2\right\}
\]
and hence
\begin{equation}\label{twovolu}
\mathscr{V}\left(\frac 14(r-c)^2\right)\le \mathrm{Vol}(B(q,r))\le\mathscr{V}\left(\frac 14(r+c)^2\right).
\end{equation}
Using present notations, \eqref{VRrel} can be rewritten as
\begin{equation}\label{en1}
0\le\frac{n}{2}\mathscr{V}(t)-\mathscr{R}(t)=t\mathscr{V}
^{\prime}(t)-\mathscr{R}^{\prime}(t).
\end{equation}
For any $t>0$, let
\[
\operatorname{P}(t):=\frac{\mathscr{V}(t)
}{t^{\frac{n}{2}}}-\frac{\mathscr{R}(t)}{t^{\frac{n}{2}+1}}
\quad\mathrm{and}\quad
\operatorname{N}(t):=\frac{\mathscr{R}(t)}{t\mathscr{V}(t)}.
\]
Then \eqref{en1} implies
\begin{equation}\label{relaPN}
\begin{aligned}
\operatorname{P}^{\prime}(t)&=-\left(  1-\frac{n+2}{2t}\right)  \frac
{\mathscr{R}(t)}{t^{\frac{n}{2}+1}}\\
&=-\frac{\left(1-\frac{n+2}{2t}\right)  \operatorname{N}(t)}
{1-\operatorname{N}(t)}\operatorname{P}(t).
\end{aligned}
\end{equation}
This implies $\operatorname{P}(t)$ is decreasing and
\begin{equation}\label{relaPV}
\left(1-\frac{n}{2t}\right) \frac{\mathscr{V}(t)}{t^{\frac{n}{2}}}
\le\operatorname{P}(t)\le\frac{\mathscr{V}(t)}{t^{\frac{n}{2}}}
\end{equation}
for $t\ge n/2+1$, where we used $\frac{\mathscr{R}(t)}{\mathscr{V}(t)}\le n/2$.
Integrating equality \eqref{relaPN} gives
\[
\operatorname{P}(t)=\operatorname{P}(n+2) e^{-\int_{n+2}^t\frac{\left(
1-\frac{n+2}{2\tau}\right)  \operatorname{N}(\tau)}
{1-\operatorname{N}(\tau)}d\tau}
\]
for all $t\ge n+2$. Since $\mathrm{S}\ge\delta$, then $\operatorname{N}(\tau)\ge\delta/\tau$.
Also noticing that $\frac{\operatorname{N}(\tau)}{1-\operatorname{N}(\tau)}$
is increasing in $\operatorname{N}(\tau)$, hence the above equality can be estimated by
\begin{equation*}
\begin{aligned}
\operatorname{P}(t)&\le\operatorname{P}(n+2) e^{-\int_{n+2}^t\left(1-\frac{n+2}{2\tau}\right)\frac{\delta}{\tau-\delta}d\tau}\\
&\le\operatorname{P}(n+2) e^{-\int_{n+2}^t\left(1-\frac{n+2}{2\tau}\right)\frac{\delta}{\tau}d\tau}\\
&=\operatorname{P}(n+2) (n+2)^{\delta}e^{\frac{\delta}{2}}e^{-\frac{n+2}{2t}\delta}\,t^{-\delta}
\end{aligned}
\end{equation*}
for all $t\ge n+2$. Combining this with \eqref{relaPV},
\begin{equation}\label{mathvup}
\mathscr{V}(t)\le c(n)\operatorname{P}(n+2) e^{-\frac{n+2}{2t}\delta}\, t^{\frac n2-\delta}
\end{equation}
for all $t\ge n+2$, where we used $\delta< n/2$. By Lemma \ref{potenesti},
since $B(q, 2\sqrt{t}-c)\subset \{f<t\}$, where $c:=2\sqrt{f(q)}+4n-4/3$,
combining \eqref{twovolu}, it follows that
\[
\mathrm{Vol}\left(B(q, 2\sqrt{t}-c)\right)\le\mathscr{V}(t)
\]
for $t\ge c^2/4$. Combining this with \eqref{mathvup} yields
\[
\mathrm{Vol}\left(B(q, 2\sqrt{t}-c)\right)\le c(n)\operatorname{P}(n+2)
e^{-\frac{n+2}{2t}\delta}\, t^{\frac n2-\delta}
\]
for all $t\ge n+2+c^2/4$, so that
\[
\mathrm{Vol}\left(B(q, r)\right)\le c(n)\operatorname{P}(n+2)
e^{-\frac{2(n+2)\delta}{(r+c)^2}}\left(\frac{r+c}{2}\right)^{n-2\delta}
\]
for all $r\ge2\sqrt{n+2}$. Noticing that
\[
\operatorname{P}(n+2)\le\frac{\mathscr{V}(n+2)}{(n+2)^{\frac{n}{2}}}
\quad\mathrm{and}\quad
\mathscr{V}(n+2)\le \mathrm{Vol}\left(B(q,2\sqrt{n+2}+c)\right), \]
then \[
\mathrm{Vol}\left(B(q, r)\right)\le c(n)\mathrm{Vol}\left(B(q,2\sqrt{n+2}+c)\right)
e^{-\frac{\delta}{r^2}}\,r^{n-2\delta}
\]
for all $r\ge2\sqrt{n+2}+c$. Therefore the second estimate follows
by applying the first estimate of Lemma \ref{logeq1}
\begin{equation*}
\begin{aligned}
\mathrm{Vol}\left(B(q,2\sqrt{n+2}+c)\right)&\le c(n)e^{f(q)}(2\sqrt{n+2}+c)^n\\
&\le c(n)e^{f(q)},
\end{aligned}
\end{equation*}
where we used a fact that
\[
e^{f(q)}(2\sqrt{n+2}+c)^n\le c(n)e^{f(q)}f(q)^{n/2}\le\widetilde{c}(n)e^{f(q)}.
\]

Finally, the third estimate of the lemma follows by the second estimate and
a basic fact $f(p)\le n/2$.
\end{proof}

\begin{remark}\label{levcon}
The above argument shows that the point-wise condition of scalar curvature in Lemma
\ref{logeq1} can be replaced by a condition of the average scalar curvature over
the level set $\{f<r\}$, that is,
\[
\frac{1}{\int_{\{f<r\}}dv}\int_{\{f<r\}}\mathrm{S}\,dv\ge\delta
\]
for any $r>0$. This is because we only used $\frac{\mathscr{R}(t)}{\mathscr{V}(t)}\ge\delta$
in the proof of Lemma \ref{logeq1}.
\end{remark}

For a lower volume estimate, a sharp version was proved by Munteanu-Wang
(see Theorem 1.6 in \cite {[MuWa12]} or Theorem 1.4 in \cite{[MuW14]}).
But coefficients of these estimates all depend on a base point, which
will be trouble in dealing with our issue. So in the following we shall
adopt a Li-Wang's local lower volume estimate for any base point, which
comes from the Sobolev inequality (see Theorem 23 in \cite{[LiWa]}). This
estimate is more useful when $r$ is sufficiently large.
\begin{lemma}\label{logeq2}
Let $(M,g, f)$ be an $n$-dimensional complete non-compact shrinker satisfying
\eqref{Eq1}, \eqref{condition} and \eqref{condmu}. For any point $q\in M$ and for any $r>0$,
\[
\frac{\mathrm{Vol}(B(q,r))}{r^n}
\left[1+\sup_{s\in[0,r]}\frac{s^2\int_{B(q,s)}\mathrm{S}\,dv}{\mathrm{Vol}(B(q,s))}\right]^{n/2}
\ge c(n)e^{\mu}.
\]
In particular, if the scalar curvature $\mathrm{S}\le\Lambda$ for some
constant $\Lambda\ge 0$ in $B(q,r)\subset M$, then
\[
\frac{\mathrm{Vol}(B(q,r))}{r^n}(1+\Lambda r^2)^{n/2}\ge c(n)e^{\mu}.
\]
\end{lemma}

\begin{proof}[Proof of Lemma \ref{logeq2}]
The argument is essentially the same as the proof of Theorem 23 in
\cite{[LiWa]}. For the reader's convience, we provide the detailed proof.
For a base point $q\in M$, we choose $r_0\in[0,r]$ such that
\[
\inf_{s\in[0,r]}\frac{\mathrm{Vol}(B(q,s))}{s^n}
\]
is attained at $r_0$. Below we discuss two cases $r_0=0$ and $r_0>0 $ separately.

Case one: $r_0=0$. We have
\[
\mathrm{Vol}(B(q,r))\ge \omega_n r^n,
\]
where $\omega_n$ is the volume of the unit Euclidean $n$-ball. Now we
claim that $\mu\le 0$. Indeed, for $\tau\to 0+$, we have that $(M^n,p,\tau^{-1}g)$
converges to Euclidean space $(\mathbb{R}^n,0,g_E)$ smoothly in the
Cheeger-Gromov sense. By Lemma 3.2 of \cite{[Liy]}, we know
\[
\underset{\tau\to 0+}{\lim\sup}\,\mu(g,\tau)
=\underset{\tau\to 0+}{\lim\sup}\,\mu(\tau^{-1}g,1)
\le \mu(g_E,1)=0.
\]
Also, since $\mu(g,\tau)\ge \mu(g,1)=\mu$ for each $\tau\in(0,1)$ by Lemma 15
in \cite{[LiWa]}, then the claim $\mu\le 0$ follows. Hence the estimate
of Case one follows.

Case two: $r_0>0$. Let $\phi:\mathbb{R}\to[0,1]$ be a smooth function such
that $\phi(t)=1$ on $(-\infty,1/2]$, $\phi(t)=0$ on $[1,+\infty)$ and
$|\phi'|\le 2$ on $[0,\infty)$. For any point $q\in M$, let
\[
u(x):=\phi\left(\frac{r(q,x)}{r_0}\right).
\]
Clearly, $u$ is supported in $B(q, r_0)$ and it satisfies $|\nabla u|\le2r_0^{-2}$.
We substitute the above special function $u$ into \eqref{sobo} of Lemma \ref{sobineq}
and get
\begin{equation*}
\begin{aligned}
\mathrm{Vol}\left(B(q,\frac{r_0}{2})\right)^{\frac{n-2}{n}}
&\le c(n)e^{-\frac{2\mu}{n}}\int_{B(q,r_0)}\left(4|\nabla u|^2+\mathrm{S}u^2\right)dv\\
&\le c(n)e^{-\frac{2\mu}{n}}\frac{\mathrm{Vol}(B(q,r_0))}{r^2_0}
\left[1+\frac{r^2_0\int_{B(q,r_0)}\mathrm{S}\,dv}{\mathrm{Vol}(B(q,r_0)}\right].
\end{aligned}
\end{equation*}
From the choice of $r_0$, we see that
\[
\mathrm{Vol}\left(B(q,\frac{r_0}{2})\right)\ge \frac{\mathrm{Vol}(B(q,r_0))}{2^n}.
\]
Combining the above two inequalities yields
\[
\frac{\mathrm{Vol}(B(q,r_0))}{r^n_0}
\left[1+\frac{r^2_0\int_{B(q,r_0)}\mathrm{S}\,dv}{\mathrm{Vol}(B(q,r_0))}\right]^{n/2}
\ge c(n)e^{\mu}.
\]
According to the definition of $r_0$, we have
\[
\frac{\mathrm{Vol}(B(q,r))}{r^n}\ge \frac{\mathrm{Vol}(B(q,r_0))}{r^n_0}.
\]
Combining the above two inequalities gives the conclusion of Case two.
\end{proof}

At the end of this section, we give another version of lower volume estimate
by using the logarithmic Sobolev inequality \eqref{LSI}, which is sharper
than Lemma \ref{logeq2} when $r$ is not sufficiently large.
\begin{lemma}\label{slogeq}
Let $(M,g, f)$ be an $n$-dimensional complete non-compact shrinker satisfying \eqref{Eq1},
\eqref{condition} and \eqref{condmu}. For any point $q\in M$,
\begin{equation}\label{LSIequ}
\mu+n+\frac n2\ln(4\pi)+16(1-2\cdot 5^n)\le
+2\cdot 5^n r^2\frac{\int_{B(q,r)}\mathrm{S}\,dv}{\mathrm{Vol}(B(q,r))}
+\ln \frac{\mathrm{Vol}(B(q,r))}{r^n}
\end{equation}
for any $r\ge 4(\sqrt{n}+c)$, where $c:=2\sqrt{f(q)}+4n-4/3$.
\end{lemma}

\begin{proof}[Proof of Lemma \ref{slogeq}]
Let $\phi:[0,\infty)\to[0,1]$
be a smooth cut-off function supported in $[0,1]$ such that $\phi(t)=1$ on
$[0,1/2]$ and $|\phi'|\le 2$ on $[0,\infty)$. For any $q\in M$ and any
$r>0$, let
\[
\varphi(x):=e^{-\theta/2}\phi\left(\frac{r(q,x)}{r}\right),
\]
where $\theta$ is some constant determined by condition $\int_M\varphi^2dv=1$.
Clearly, $\varphi$ is supported in $B(p,r)$ and it satisfies
$|\nabla\varphi|\le 2r^{-1}\cdot e^{-\theta/2}$. Moreover, $\theta$ satisfies
\[
\mathrm{Vol}\left(B(q,\frac r2)\right)\le e^{\theta}\int_M\varphi^2dv=e^{\theta}
\]
and
\[
e^{\theta}=e^{\theta}\int_M\varphi^2dv
=\int_M\phi^2\left(\frac{r(q,x)}{r}\right)dv\le\mathrm{Vol}(B(q,r)).
\]

Now we shall substitute the above cut-off function $\varphi$ into Lemma 
\ref{sobineq} to simplify the inequality \eqref{LSI}.

First, by the definition of $\varphi$ and lower bound of $e^{\theta}$, we have
\begin{equation}\label{est1}
\begin{aligned}
4\tau\int_M |\nabla\varphi|^2dv&=4\tau\int_{B(q,r)\backslash B(q,\frac r2)} |\nabla\varphi|^2dv\\
&\le\frac{16\tau}{r^2}\left[\mathrm{Vol}(B(q,r))-\mathrm{Vol}\left(B(q,\frac r2)\right)\right]e^{-\theta}\\
&\le\frac{16\tau}{r^2}\left[\frac{\mathrm{Vol}(B(q,r))}{\mathrm{Vol}\left(B(q,\frac r2)\right)}-1\right]\\
&\le 16(2\cdot 5^n-1)\frac{\tau}{r^2}
\end{aligned}
\end{equation}
for all $r\ge 4(\sqrt{n}+c)$, where $c:=2\sqrt{f(q)}+4n-4/3$. In the last
inequality, we used Theorem \ref{relcompar} in the following form:
\[
\frac{\mathrm{Vol}(B(q,r))}{\mathrm{Vol}(B(q,\frac r2))}\le 2\cdot5^n
\]
for any $r\ge 4(\sqrt{n}+c)$.

Second, by the definition of $\varphi$ and the lower bound of $e^{\theta}$,
we have the estimate
\begin{equation}\label{est2}
\begin{aligned}
\tau\int_M\mathrm{S}\varphi^2 dv&\le \tau e^{-\theta}\int_{B(q,r)}\mathrm{S}\,dv\\
&\le\frac{\tau}{\mathrm{Vol}\left(B(q,\frac r2)\right)}\int_{B(q,r)}\mathrm{S}\,dv\\
&\le2\cdot5^n\frac{\tau\int_{B(q,r)}\mathrm{S}\,dv}{\mathrm{Vol}(B(q,r))}
\end{aligned}
\end{equation}
for all $r\ge 4(\sqrt{n}+c)$, where $c:=2\sqrt{f(q)}+4n-4/3$. Here we still used
Theorem \ref{relcompar} in the above last inequality.

Third, we will apply the Jensen's inequality to estimate the term:
$-\int_M\varphi^2\ln \varphi^2dv$. Since smooth function $H(t):=-t\ln t$ is concave
in $t>0$ and the Riemannian measure $dv$ is supported in $B(q,r)$, by
the following Jensen's inequality
\[
\frac{\int H(\varphi^2)dv}{\int dv}\leq H\left(\frac{\int \varphi^2 dv}{\int dv}\right)
\]
and the definition of $H(t)$, we obtain
\[
-\frac{\int_{B(q,r)}\varphi^2\ln\varphi^2dv}{\int_{B(q,r)}dv}
\leq-\frac{\int_{B(q,r)}\varphi^2dv}{\int_{B(q,r)}dv}\ln\left(\frac{\int_{B(q,r)}\varphi^2dv}{\int_{B(q,r)}dv}\right).
\]
Since $\int_{B(q,r)}\varphi^2dv=1$, we further have a simple form
\[
-\int_{B(q,r)}\varphi^2\ln\varphi^2dv\le\ln \mathrm{Vol}(B(q,r)).
\]
Therefore,
\begin{equation}\label{est3}
\begin{aligned}
-\int_M\varphi^2\ln\varphi^2dv&=-\int_{B(q,r)}\varphi^2\ln\varphi^2dv\\
&\le\ln \mathrm{Vol}(B(q,r)).
\end{aligned}
\end{equation}

Now we substitute \eqref{est1}, \eqref{est2} and \eqref{est3} into \eqref{LSI}
and get that
\[
\mu+n+\frac n2\ln(4\pi)\le 16(2\cdot 5^n-1)\frac{\tau}{r^2}
+2\cdot5^n\frac{\tau\int_{B(q,r)}\mathrm{S}\,dv}{\mathrm{Vol}(B(q,r))}
+\ln \frac{\mathrm{Vol}(B(q,r))}{\tau^{\frac n2}}
\]
for any $\tau>0$ and for any $r\ge 4(\sqrt{n}+c)$, where $c:=2\sqrt{f(q)}+4n-4/3$.
Finally we let $\tau=r^2$ and the result follows.
\end{proof}

%
\section{Ends on a general shrinker}\label{sec3}
In this section, we will give a weak ball covering property depending
on the radius of a general shrinker without any assumption. Then we will
apply the weak ball covering to prove Theorem \ref{endest}. With the
help of Lemmas \ref{logeq1} and \ref{logeq2}, we first establish a
weak volume comparison condition on shrinkers.
\begin{proposition}\label{vd}
Let $(M,g, f)$ be an $n$-dimensional complete non-compact shrinker with a
infimum point $p\in M$ of $f$ satisfying \eqref{Eq1}, \eqref{condition} and
\eqref{condmu}. If the scalar curvature $\mathrm{S}\ge\delta$ for some
constant $\delta\ge 0$, then for any $r\ge c(n)$ and for any
$x\in \overline{B(p,2r)}$,
\[
\frac{\mathrm{Vol}(B(p,2r))}{\mathrm{Vol}\left(B(x,\tfrac{r}{8})\right)}\le c(n)e^{-\mu}r^{2(n-\delta)}.
\]
In addition, if the scalar curvature $\mathrm{S}\le \sigma$
for some constant $\sigma\ge \delta$ in $M$, then
\[
\frac{\mathrm{Vol}(B(p,2r))}{\mathrm{Vol}\left(B(x,\tfrac{r}{8})\right)}\le c(n)e^{-\mu}\sigma^{n/2}r^{n-2\delta}
\]
for any point $x\in M$ and for any $r\ge c(n)$.
\end{proposition}
\begin{proof}[Proof of Proposition \ref{vd}]
By Lemma \ref{logeq1}, we have
\begin{equation}\label{upp}
\mathrm{Vol}(B(p,2r))\le c(n)r^{n-2\delta}
\end{equation}
for any $r\ge c(n)$. On the other hand, the second estimate of Lemma \ref{logeq2}
shows that
\begin{equation}\label{low}
\mathrm{Vol}(B(x,r))\ge c(n)e^{\mu}r^n(1+\Lambda r^2)^{-n/2}
\end{equation}
with $\mathrm{S}\le\Lambda$ in $B(x,r)\subset M$. Now we want to find an upper
bound of scalar curvature $\mathrm{S}$ in $B(x,r)$. From \eqref{scaup}, we know
\[
\mathrm{S}(y)\le\frac 14\left(r(y,p)+\sqrt{2n}\right)^2
\]
for all $y\in B(x,r)$. Since $x\in\overline{B(p,2r)}$, by the triangle inequality,
we further have
\begin{equation*}
\begin{aligned}
\mathrm{S}(y)&\le\frac 14\left(r(y,x)+r(x,p)+\sqrt{2n}\right)^2\\
&\le\frac 14\left(r+2r+\sqrt{2n}\right)^2\\
&\le\frac 14(3+\sqrt{2})^2r^2
\end{aligned}
\end{equation*}
for all $y\in B(x,r)$ and for all $r\ge\sqrt{n}$. Substituting this into \eqref{low} yields
\[
\mathrm{Vol}(B(x,r))\ge c(n)e^{\mu}r^{-n}
\]
for all $r\ge\sqrt{n}$. Combining this with \eqref{upp} immediately yields the first estimate of theorem.

Next we will prove the second part of theorem. Since we also assume that
$\mathrm{S}\le \sigma$ for some constant $\sigma\ge\delta$ in $M$,
substituting this into \eqref{low}, we have
\[
\mathrm{Vol}(B(x,r))\ge c(n)e^{\mu}\sigma^{-n/2}
\]
for any point $x\in M$ and any $r\ge 1$. Combining this with \eqref{upp} gives the second
estimate.
\end{proof}

Inspired by Liu's argument \cite{[Liu2]}, we shall apply Proposition \ref{vd} to
give a weak ball covering property for sufficiently large balls in a shrinker without
any assumption. Our argument will be focused on a sufficiently large fixed radius.
\begin{theorem}\label{sdest}
Let $(M,g,f)$ be a complete non-compact shrinker with a infimum point $p\in M$ of $f$
satisfying \eqref{Eq1}, \eqref{condition} and \eqref{condmu}. If the scalar curvature
$\mathrm{S}\ge\delta$ for some constant $\delta\ge 0$, then for sufficiently large
$r\ge c(n)$, there exists 
\[
N=c(n)e^{-\mu}r^{2(n-\delta)}
\]
such that we can find points
$p_1,\ldots, p_k\in B(p,2r)\backslash\overline{B(p,r)}$, where $k=k(r)\le N$, with
\[
\bigcup^k_{i=1}B\left(p_i,\frac{r}{4}\right)\supset B(p,2r)\backslash\overline{B(p,r)}.
\]
\end{theorem}
\begin{proof}[Proof of Theorem \ref{sdest}]
For a sufficiently large fixed $r\ge c(n)$, we let
$k:=k(r)$ denote the maximum number of disjoint geodesic balls of radius
$r/8$ with centers $p_1,\ldots, p_k$ in $B(p,2r)\backslash\overline{B(p,r)}$.
Obviously, in this case,
\[
\bigcup^k_{i=1}B\left(p_i,\frac{r}{4}\right)\supset B(p,2r)\backslash\overline{B(p,r)}.
\]
See Figure 1 for a detailed description.
\begin{figure}
    \centering
\includegraphics[scale=0.6]{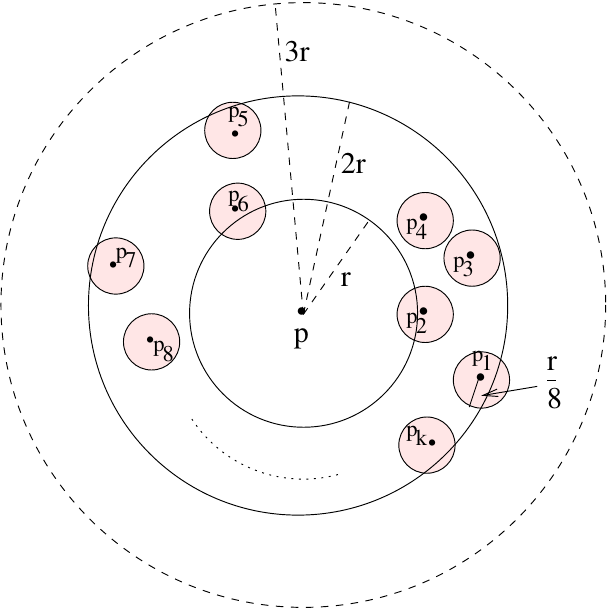}
\caption{ \footnotesize Annulus is covered by small balls}
\label{Fig1}
\end{figure}

Since $p_i\in B(p,2r)\backslash\overline{B(p,r)}$, we may let $p_i\in\partial B(p,\beta_i r)$
for some $1<\beta_i<2$, where $i=1,\ldots,k$. By the first estimate of Proposition \ref{vd}, we have
\begin{equation*}
\begin{aligned}
\mathrm{Vol}(B(p,\beta_i r))&\le\mathrm{Vol}(B(p,2r))\\
&\le c(n)e^{-\mu}r^{2(n-\delta)}\mathrm{Vol}\left(B(p_i,\frac{r}{8})\right)
\end{aligned}
\end{equation*}
for $r\ge c(n)$. By Theorem \ref{relcompar}, we also have
\[
\mathrm{Vol}(B(p,3r))\le2\left(1+\frac{6}{\beta_i}\right)^n\mathrm{Vol}\left(B(p,\beta_i r)\right)
\]
for $r\ge2(\sqrt{n}+c)$, where $c:=2\sqrt{n/2}+4n-4/3$ and $i=1,\ldots,k$. Combining the above
two estimates, for each $i$,
\[
\mathrm{Vol}(B(p,3r))\le c(n)e^{-\mu}r^{2(n-\delta)} \mathrm{Vol}\left(B(p_i,\frac{r}{8})\right)
\]
for $r\ge c(n)$, where we used $1<\beta_i<2$. Summing the above $k$ inequalities,
we get
\[
k(r)\mathrm{Vol}(B(p,3r))\le c(n)e^{-\mu}r^{2(n-\delta)}
\sum^k_{i=1}\mathrm{Vol}\left(B(p_i,\frac{r}{8})\right)
\]
for $r\ge c(n)$. On the other hand, we easily see that
\[
\sum^k_{i=1}\mathrm{Vol}\left(B(p_i,\frac{r}{8})\right)\le \mathrm{Vol}(B(p,3r)).
\]
Combining the above two estimates gives
\[
k(r)\le c(n)e^{-\mu}r^{2(n-\delta)}
\]
for $r\ge c(n)$, which completes the proof.
\end{proof}
\begin{remark}\label{reN1}
In Theorem \ref{sdest}, if the scalar curvature also satisfies $\mathrm{S}\le\sigma$
for some constant $\sigma\ge\delta$ in $M$, then for a sufficiently large $r$, we
can choose an $(n-2\delta)$-degree as follows:
\[
N=c(n)e^{-\mu}\sigma^{n/2}r^{n-2\delta}.
\]
\end{remark}

The above weak ball covering property immediately implies Theorem \ref{endest}.
\begin{proof}[Proof of Theorem \ref{endest}]
Let $(M,g,f)$ be an $n$-dimensional complete non-compact shrinker satisfying
\eqref{Eq1}, \eqref{condition} and \eqref{condmu}. Since the number of ends on
the shrinker is independent of the choice of the base point, we can choose a
infimum point $p$ of $f$ as a base point in $M$.

Given a sufficiently large fixed number $r$, let
\[
N_1=c(n)e^{-\mu}r^{2(n-\delta)}
\]
as in Theorem \ref{sdest}. That is we can find points $p_1,\ldots, p_k\in B(p,2r)\backslash\overline{B(p,r)}$, where $k=k(r)\le N_1$, with
\[
\bigcup^k_{i=1}B\left(p_i,\frac{r}{4}\right)\supset B(p,2r)\backslash\overline{B(p,r)}.
\]
Next we will prove Theorem \ref{endest} by a contradiction argument.

If Theorem \ref{endest} is not true, that is, the number of ends grows faster
than polynomial growth with degree $2(n-\delta)$, then for the above mentioned
sufficiently large $r$, there exists more than
\[
\widetilde{N}_1=c(n)e^{-\mu}r^{2(n-\delta)+\epsilon},
\]
where $\epsilon>0$ is any small constant, unbounded ends $E_j$ with respect to
$\overline{B(p,r)}$.

It is obvious that geodesic balls of radius $r/4$ with centers in different
components $E_j\cap B(p,2r)$ do not intersect. Thus we need at least
$\widetilde{N}_1$ geodesic balls of radius $r/4$ to cover the sets
$E_j\cap B(p,2r)\subset B(p,2r)\backslash\overline{B(p,r)}$, which
contradicts Theorem \ref{sdest}.
\end{proof}

\begin{remark}\label{reN2}
For Theorem \ref{endest}, if the scalar curvature $\mathrm{S}\le\sigma$ for some
constant $\sigma\ge\delta$ in $M$, then we can apply Remark \ref{reN1} to the above
argument and get the same conclusion whereas the degree $2(n-\delta)$ of
polynomial growth can be reduced to $n-2\delta$.
\end{remark}

%
\section{Ends with volume comparison condition}\label{sec4}
In this section we will discuss the finite number of ends when the shrinker satisfies
volume comparison condition. In this case we first give a ball covering property,
which is similar to the manifold case of nonnegative Ricci curvature.

\begin{theorem}\label{shrendest}
Let $(M,g,f)$ be an $n$-dimensional complete non-compact shrinker with a base
point $q\in M$ satisfying volume comparison condition. There exists a constant
\[
N=N(n,\eta)
\]
depending only on $n$ and $\eta$ such that for any $r\ge2(\sqrt{n}+c)+r_0$,
where $c:=2\sqrt{f(q)}+4n-4/3$, we can find
$p_1,\ldots, p_k\in B(q,2r)\backslash\overline{B(q,r)}$, $k\le N$, with
\[
\bigcup^k_{i=1}B\left(p_i,\frac{r}{4}\right)\supset B(q,2r)\backslash\overline{B(q,r)}.
\]
\end{theorem}
\begin{proof}[Proof of Theorem \ref{shrendest}]
Let $k$ be the maximum number of disjoint geodesic
balls of radius $r/8$ with centers $p_1,\ldots, p_k$ in
$B(q,2r)\backslash\overline{B(q,r)}$. Here we choose $r$ sufficiently large such
that $r\ge2(\sqrt{n}+c)+r_0$. Clearly,
\[
\bigcup^k_{i=1}B\left(p_i,\frac{r}{4}\right)\supset B(q,2r)\backslash\overline{B(q,r)}.
\]
Since $p_i\in B(q,2r)\backslash\overline{B(q,r)}$, we may let $p_i\in\partial B(q,\beta_i r)$
for some constant $1<\beta_i<2$, where $i=1,\ldots,k$. By the volume comparison condition,
we have
\begin{equation*}
\begin{aligned}
\mathrm{Vol}(B(q,\beta_i r))&\le\eta \mathrm{Vol}\left(B(p_i,\frac{\beta_ir}{16})\right)\\
&\le\eta \mathrm{Vol}\left(B(p_i,\frac{r}{8})\right)
\end{aligned}
\end{equation*}
for all $r\ge r_0$. By Theorem \ref{relcompar}, we see that
\[
\mathrm{Vol}(B(q,3r))\le2\left(1+\frac{6}{\beta_i}\right)^n\mathrm{Vol}(B(q,\beta_i r))
\]
for $r\ge2\beta_i^{-1}(\sqrt{n}+c)$, where $i=1,\ldots,k$. Combining the above
two estimates, for each $i$, there exists a constant $C(n,\eta)$ depending only on $n$
and $\eta$ such that
\[
\mathrm{Vol}(B(q,3r))\le C(n,\eta)
\mathrm{Vol}\left(B(p_i,\frac{r}{8})\right)
\]
for $r\ge2(\sqrt{n}+c)+r_0$, where $c:=2\sqrt{f(q)}+4n-4/3$ and we used
$1<\beta_i<2$, where $i=1,\ldots,k$. This implies
\[
k\mathrm{Vol}(B(q,3r))\le C(n,\eta)
\sum^k_{i=1}\mathrm{Vol}\left(B(p_i,\frac{r}{8})\right)
\]
for $r\ge2(\sqrt{n}+c)+r_0$. On the other hand,
\[
\sum^k_{i=1}\mathrm{Vol}\left(B(p_i,\frac{r}{8})\right)\le \mathrm{Vol}(B(q,3r))
\]
Combining the above two inequalities yields $k\le C(n,\eta)$
and the result follows.
\end{proof}

Similar to the preceding discussion in Section \ref{sec4}, we can apply Theorem
\ref{shrendest} to prove Theorem \ref{main1}. Here we include it for the
completeness.
\begin{proof}[Proof Theorem \ref{main1}]
Under the assumption of Theorem \ref{main1}, we let $N_2=N(n,\eta)$
as in Theorem \ref{shrendest}. If Theorem \ref{main1} is not true, we can
take $r$ large enough such that there exist more than $N_2$ unbounded ends
$E_j$ with respect to $\overline{B(q,r)}$.

Because $E_j\cap B(q,2r)$ lie in $B(q,2r)\backslash\overline{B(q,r)}$ and geodesic
balls of radius $r/4$ with centers in different components $E_j\cap B(q,2r)$
do not intersect. That is, we need more than $N_2$ geodesic balls of radius
$r/4$ to cover $E_j\cap B(q,2r)$, which contradicts Theorem \ref{shrendest}.
\end{proof}

In the rest of this section, we will discuss four sufficient assumptions such that a
class of shrinkers satisfies volume comparison condition. As we all know, if
$(M,g)$ has nonnegative Ricci curvature everywhere, then it satisfies the
volume comparison condition. Indeed the volume doubling property sufficiently
leads to volume comparison condition.
\begin{proposition}\label{voldoub}
Let $(M,g)$ be an $n$-dimensional complete manifold satisfying the volume
doubling property. Then for all $0<r<R<\infty$ and all $x\in M$ and
$y\in\overline{B(x,R)}$,
\[
\frac{\mathrm{Vol}(B(x,R))}{\mathrm{Vol}(B(y,r))}\le D^2\left(\frac{R}{r}\right)^\kappa,
\]
where $\kappa=\log_2D$. In particular, $(M,g)$ satisfies volume comparison condition.
\end{proposition}
\begin{proof}[Proof of Proposition \ref{voldoub}]
Assume $(M,g)$ satisfies the volume doubling property, that is
\[
\mathrm{Vol}(B(x,2r))\le D\,\mathrm{Vol}(B(x,r))
\]
for any $x\in M$ and $r>0$, where $D$ is a fixed constant.
Let $m$ be a positive integer such that $2^m<R/r\le 2^{m+1}$.
Since
\[
B(x,R)\subset B(y,2R)\subset B(y,2^{m+2}r)
\]
and thus
\[
\mathrm{Vol}(B(x,R))\le\mathrm{Vol}(B(y,2^{m+2}r)),
\]
then we have
\begin{equation*}
\begin{aligned}
\mathrm{Vol}(B(x,R))&\le D^{m+2}\mathrm{Vol}(B(y,r))\\
&\le D^2\left(\frac{R}{r}\right)^{\kappa}\mathrm{Vol}(B(y,r)),
\end{aligned}
\end{equation*}
where $\kappa=\log_2D$. This proves the first estimate.

In particular, when $y\in \partial B(x,R)$, we let $r=R/16$ in the first
estimate and immediately get volume comparison condition.
\end{proof}

Second, we observe that the shrinker with at least quadratic decay of scalar
curvature implies some non-collapsed property and hence satisfies volume comparison
condition.
\begin{proposition}\label{decc}
Let $(M,g,f)$ be a complete non-compact shrinker with a infimum point $p\in M$ of $f$
satisfying \eqref{Eq1}, \eqref{condition} and \eqref{condmu}. If the scalar
curvature satisfies
\[
\mathrm{S}(x)\cdot r^2(p,x)\le c_0
\]
for any $r(p,x)>0$, where $c_0>0$ is a constant and $r(p,x)$ is the distance function
from $p$ to $x$, then the shrinker satisfies volume comparison condition.
In particular, any shrinker with finite asymptotic scalar curvature ratio
satisfies volume comparison condition.
\end{proposition}
\begin{proof}[Proof of Proposition \ref{decc}]
For any $1/32\le\alpha\le 1/2$, for any
$r>0$ and for any point $q\in \partial B(p,r)$, by the second estimate of
Lemma \ref{logeq2}, we have
\[
(\alpha r)^{-n}\mathrm{Vol}(B(q,\alpha r))\ge c(n)e^{\mu}
\left[1+\frac{c_0}{(1-\alpha)^2r^2}\cdot(\alpha r)^2\right]^{-\frac n2}
\]
where we used
\[
\mathrm{S}\le\frac{c_0}{r^2(p,x)}\le\frac{c_0}{(1-\alpha)^2r^2}.
\]
Namely, for any $r>0$ and for any point $q\in \partial B(p,r)$,
\begin{equation*}
\begin{aligned}
\mathrm{Vol}(B(q,\alpha r))&\ge c(n)e^{\mu}\left[1+\frac{c_0\alpha^2}{(1-\alpha)^2}\right]^{-\frac n2}\alpha^n\cdot r^n\\
&\ge c(n,c_0)e^{\mu}r^n
\end{aligned}
\end{equation*}
for some constant $c(n,c_0)$ depending only on $n$ and $c_0$, where used $1/32\le\alpha\le 1/2$.

On the other hand, by Lemma \ref{logeq1},
\[
\mathrm{Vol}(B(p,r))\le c(n)r^n
\]
for any $r>0$. Thus, for any $r>0$ and for any point $q\in \partial B(p,r)$, the
lower and upper volume estimates give
\[
\frac{\mathrm{Vol}(B(p,r))}{\mathrm{Vol}(B(q,\alpha r))}\le c(n,c_0)e^{-\mu}.
\]
Letting $\alpha=1/16$ shows that such shrinker satisfies
volume comparison condition.
\end{proof}

The proof of Proposition \ref{decc} indicates that the finite asymptotic scalar
curvature ratio implies the positive asymptotic volume ratio. Moreover,
combining Proposition \ref{decc} and Theorem \ref{main1}, we easily get the
following result due to Munteanu, Schulze and Wang \cite{[MSW]}.
\begin{corollary}\label{cor}
Any complete non-compact shrinker with finite asymptotic scalar
curvature ratio must have finitely many ends.
\end{corollary}

Third, we see that if a family of the average of scalar curvature integral
has at least quadratic decay of radius, then such shrinker also satisfies
volume comparison condition.
\begin{proposition}\label{intevc}
Let $(M,g,f)$ be a complete non-compact shrinker with a infimum point $p\in M$ of $f$
satisfying \eqref{Eq1}, \eqref{condition} and \eqref{condmu}. If
there exists a constant $c_1>0$ such that
\begin{equation}\label{intsca}
\frac{r^2}{\mathrm{Vol}\left(B(x,r)\right)}\int_{B(x,r)}\mathrm{S}\,dv\le c_1
\end{equation}
for all $r>0$ and all $x\in\partial B(p,r)$, then the shrinker satisfies volume comparison condition.
\end{proposition}

\begin{proof}[Proof of Proposition \ref{intevc}]
For any $r>0$, we let point $q$ be $x\in\partial B(p,r)$ in the first estimate
of Lemma \ref{logeq2}, and get
\[
\frac{\mathrm{Vol}\left(B(x,\frac{r}{16})\right)}{(\tfrac{r}{16})^n}
\left[1+\sup_{s\in\left[0,\tfrac{r}{16}\right]}\frac{s^2\int_{B(x,s)}\mathrm{S}\,dv}{\mathrm{Vol}(B(x,s))}\right]^{n/2}
\ge c(n)e^{\mu}.
\]
By the assumption \eqref{intsca}, the above inequality becomes
\[
\mathrm{Vol}\left(B(x,\frac{r}{16})\right)\ge c(n,c_1)e^{\mu}r^n
\]
for all $r>0$ and all $x\in\partial B(p,r)$. Combining this with the
volume upper growth $\mathrm{Vol}(B(p,r))\le c(n)r^n$
immediately yields
\[
\frac{\mathrm{Vol}(B(p,r))}{\mathrm{Vol}\left(B(x,\frac{r}{16})\right)}\le c(n,c_1)e^{-\mu}
\]
for any $r>0$ and all $x\in \partial B(p,r)$.
\end{proof}
\begin{remark}
Similar to the above argument, Proposition \ref{intevc} can be also proved by
Lemma \ref{slogeq}. Moreover, when $n\ge 3$, the assumption \eqref{intsca}
in Proposition \ref{intevc} can be replaced by the bound of the following maximal
function of scalar curvature introduced by Topping \cite{[To]}:
\[
\sup_{s\in\left(0,\tfrac{r}{16}\right]}s^{-1}
\left[\mathrm{Vol}(B(x,s))\right]^{-\frac{n-3}{2}}\left(\int_{B(x,s)} \mathrm{S}\,dv\right)^{\frac{n-1}{2}}\le\delta,
\]
for all $r>0$ and all $x\in\partial B(p,r)$,
where $\delta:=\min\{w_n,\,(4\pi)^{\frac n2}e^{\mu+n-2^n\cdot17}\}$
and $\omega_n$ is the volume of the unit Euclidean $n$-ball. This bound
assumption also enables us to get that
\[
\mathrm{Vol}\left(B(x,\frac{r}{16})\right)> \delta\,r^n
\]
for all $r>0$ and all $x\in\partial B(p,r)$, the interested readers
are referred to Theorem 3.1 of \cite{[Wu21]} for detailed proof.
\end{remark}

Combining Proposition \ref{intevc} and Theorem \ref{main1} leads to
\begin{corollary}\label{intcor}
Any complete non-compact shrinker satisfying \eqref{intsca}
must have finitely many ends.
\end{corollary}

In the proof of Corollaries \ref{cor} and \ref{intcor}, we observe that
these curvature assumptions both imply a family of Euclidean volume growth.
These proof indeed shows that any shrinker with a family of Euclidean volume
growth must have volume comparison condition.
\begin{corollary}\label{AVRc}
If a complete non-compact shrinker $(M,g,f)$ with a infimum point $p\in M$ of $f$
satisfies
\begin{equation}\label{famieq}
\mathrm{Vol}\left(B(x,\frac{r}{16})\right)\ge c\,r^n
\end{equation}
for all $r\ge r_0$ for some $r_0>0$, and all $x\in\partial B(p,r)$, where $c$ is a positive constant
independent of $x$ and $r$, then such shrinker satisfies volume comparison
condition and hence has finitely many ends.
\end{corollary}

In the end of this section, we give some comments on the relation between
Corollary \ref{AVRc} and asymptotic volume ratio on shrinkers. Recall that
the \textit{asymptotic volume ratio} ($\operatorname{AVR}$) of a complete
Riemannian manifold $(M,g)$ is defined by
\[
\operatorname{AVR}(g):=\lim_{r\rightarrow\infty}\frac{\operatorname{Vol}B(q,r)}{\omega_nr^n}
\]
if the limit exists. Whenever the $\operatorname{AVR}(g)$ exists, it is independent
of point $q$. If $(M,g)$ has nonnegative Ricci curvature, then the limit always
exists by the Bishop-Gromov volume comparison. For any shrinker, Chow, Lu and Yang
\cite{[CLY]} proved that $\operatorname{AVR}(g)$ always exists and is finite. The
assumption \eqref{famieq} naturally implies positive asymptotic volume ratio;
but the reverse problem is not clear to the author at present. Notice that Feldman,
Ilmanen and Knopf \cite{[FIK]} described examples of complete non-compact K\"ahler
shrinkers, which have $\operatorname{AVR}(g)>0$ and the Ricci curvature changes sign.
We see that positive asymptotic volume ratio provides the Euclidean volume growth
based on a fixed point, which does not seem to yield a family of Euclidean volume growth
\eqref{famieq}. On the other hand, Carrillo and Ni \cite{[CaNi]} proved that any
shrinker with Ricci curvature $\mathrm{Ric}(g)\ge0$ must have $\operatorname{AVR}(g)=0$.
Here we may reverse the process and naively ask that if $\operatorname{AVR}(g)=0$ implies $\mathrm{Ric}(g)\ge0$?

\section{Diameter growth of ends}\label{sec5}

In the last section, we will apply the ball covering property to study the
diameter growth of ends in the shrinker. The manifold case can be referred to
\cite{[AG]}, where Abresch and Gromoll proved that every end of manifolds with
nonnegative Ricci curvature has most linear diameter growth. Later this result
can be generalized by Liu \cite{[Liu2]} to manifolds with nonnegative Ricci
curvature outside a compact set. Let us first recall the definition diameter
of ends on manifolds; see also \cite{[Liu2]}.

\begin{definition}
Let $q$ be a fixed point in a Riemannian manifold $(M,g)$. For any $r>0$,
any connected component $\Sigma$ of the annulus
\[
A_q(2r,\tfrac{3}{4}r):=B(q,2r)\backslash\overline{B(q,\tfrac{3}{4}r)},
\]
and any two points $x,y\in \Sigma\cap\partial B(q,r)$, we let
\[
d_r(x,y):=\inf\left\{\mathrm{length}(\gamma)\right\},
\]
where the infimum is taken over all piecewise smooth curves $\gamma$
from $x$ to $y$ in $M\backslash\overline{B(q,r/2)}$. Then we set
\[
\mathrm{diam}\left(\Sigma\cap\partial B(q,r)\right)
:=\sup_{x,y\in\Sigma\cap\partial B(q,r)}d_r(x,y).
\]
Using the above notations, the \textit{diameter of ends} at $r$ from $q$ is defined by
\[
\mathrm{diam}_q(r):=\sup_{\Sigma\subset A_q(2r,\tfrac{3}{4}r)}
\mathrm{diam}\left(\Sigma\cap\partial B(q,r)\right).
\]
See Figure 2 for a simple description.
\end{definition}

\begin{figure}
    \centering
\includegraphics[scale=0.5]{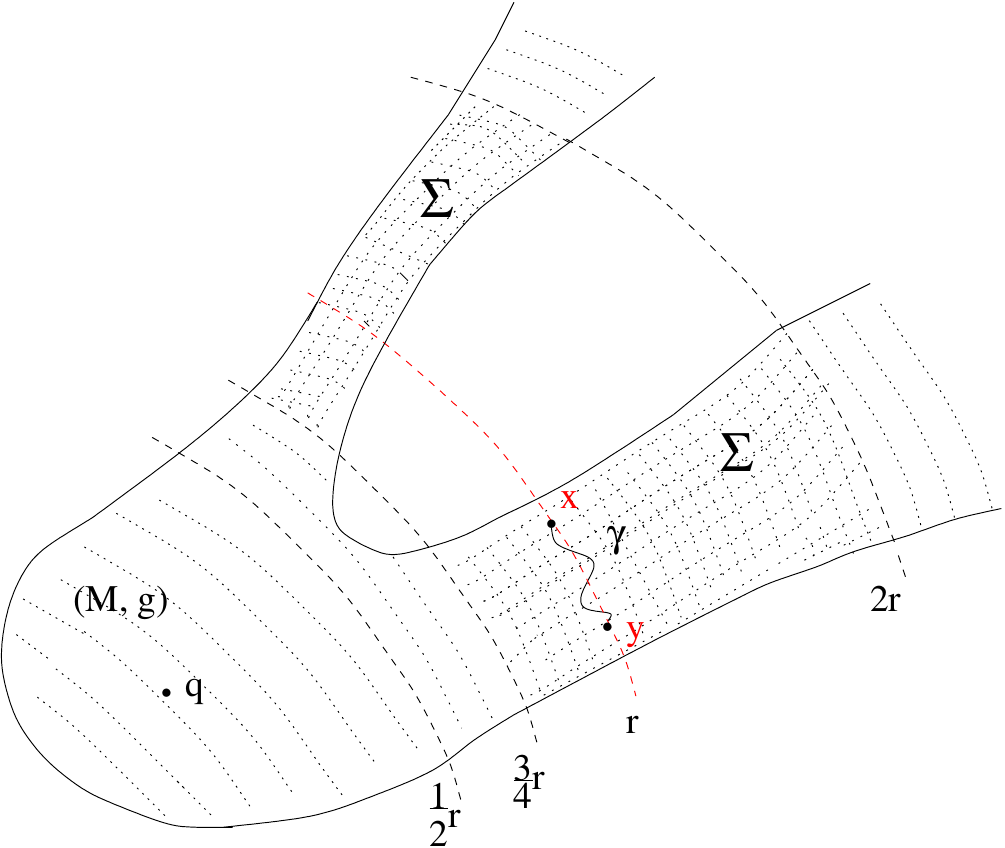}
\caption{ \footnotesize Definition of the diameter of ends}
\label{Fig2}
\end{figure}


We now apply the above definition to Theorem \ref{diam} and obtain a diameter
growth for ends in the shrinker without any assumption.
\begin{theorem}\label{diam}
On any $n$-dimensional complete non-compact shrinker with the scalar curvature
\[
\mathrm{S}\ge \delta
\]
for some constant $\delta\ge 0$, the diameter growth of
ends is at most polynomial growth with degree $2(n-\delta)+1$.
\end{theorem}
\begin{proof}[Proof of Theorem \ref{diam}]
Without loss of generality, we choose a infimum point $p\in M$ of $f$ as a base point.
By Theorem \ref{sdest}, for a fixed sufficiently large $r$, and for any connected component
$\Sigma$ of the annulus $A_q(2r,\tfrac{3}{4}r)$, we can find no more than
\[
N:=c(n)e^{-\mu}r^{2(n-\delta)}
\]
geodesic balls $B_i:=B\left(p_i,\frac{r}{4}\right)$, where
$p_i\in A_q(2r,\tfrac{3}{4}r)$ and $i\le N$ such that
\[
\bigcup_{i=1}B\left(p_i,\frac{r}{4}\right)\supset \Sigma.
\]
For any two points $x$ and $y$ in $\Sigma\cap\partial B(q,r)$,
since $\Sigma$ is connected, we can find a subsequence of geodesic balls $\{B_i\}$:
$B_{i_1},\ldots,B_{i_k}$, where $k\le N$ such that
\[
x\in B_{i_1},\quad B_{i_j}\cap B_{i_{j+1}}\neq \emptyset\,\,(j=1 ,\ldots,k-1), \quad y\in B_{i_k}.
\]
Now we choose fixed points $z_j\in B_{i_j}\cap B_{i_{j+1}}$ and consecutively connect
the above mentioned points
\[
x,p_{i_1},z_1,p_{i_2},z_2,p_{i_3},\ldots,p_{i_{k-1}},z_{k-1},p_{i_k},y,
\]
which forms a piecewise smooth curve $\gamma$. Obviously, the curve $\gamma$ lies in
$M\backslash\overline{B(q,r/2)}$ and has the length of $\gamma$
\[
\mathrm{length}(\gamma)\le 2k\cdot\frac{r}{4}
\le\frac{N}{2}r\le c(n)e^{-\mu}r^{2(n-\delta)+1}.
\]
This completes the proof.
\end{proof}

\begin{remark}
If the scalar curvature of shrinker is uniformly bounded, by Remark \ref{reN1},
the above argument indicates that the degree $2(n-\delta)+1$ in Theorem
\ref{diam} can be reduced to $n-2\delta+1$.
\end{remark}

If the shrinker satisfies volume comparison condition, by the same
argument as above, Theorem \ref{shrendest} immediately implies
\begin{theorem}
On any complete non-compact shrinker with volume comparison condition,
the diameter growth of ends is at most linear.
\end{theorem}

\bibliographystyle{amsplain}

\end{document}